\documentclass[11pt,a4paper]{article}
\usepackage[OT1]{fontenc}
\usepackage[latin1]{inputenc}
\usepackage[english]{babel}
\usepackage{amsfonts}
\usepackage{amsthm}
\usepackage{amsmath}
\usepackage{graphicx}

\usepackage{amssymb}
\usepackage{dsfont}

\usepackage{latexsym}
\def\ddd{ {\bf \Delta}}
\def\h{ {\cal H} }
\def\a{ {\cal A} }

\def\l{ {\cal L} }
\def\n{ {\cal N} }

\def\b{ {\cal B} }
\def\u{ {\cal U} }
\def\m{ {\cal M} }

\def\t{ {\cal T} }
\def\s{ {\cal S} }
\def\e{ {\cal E} }

\def\p{ {\cal P} }

\def\k{ {\cal K} }

\def\z{ {\cal Z} }

\def\mumu{{\bf \mathfrak{m}}}
\def\pepe{{\bf \mathfrak{p}}}

\def\noi{\noindent}

\def\GG{\mathbb{G}}
\def\FF{{\bf F}}
\def\LL{{\bf L}}

\newtheorem{teo}{Theorem}[section]
\newtheorem{prop}[teo]{Proposition}
\newtheorem{lem}[teo]{Lemma}
\newtheorem{coro}[teo]{Corollary}
\newtheorem{defi}[teo]{Definition}

\theoremstyle{definition}
\newtheorem{rem}[teo]{Remark}

\title{A note on  common complements}
\author{Esteban Andruchow and Eduardo Chiumiento}

\begin{document}

\maketitle

\begin{abstract}
We discuss the structure of the set $\Delta$ consisting of pairs of closed subspaces  that have a common complement in a Hilbert space previously studied by Lauzon and Treil (J. Funct. Anal. 212: 500--512, 2004). 
We prove that $\Delta$ is the base space of a real analytic fiber bundle constructed in terms of geometric objects associated to the Grassmann manifold.   As a  consequence we determine the homotopy type of $\Delta$.
\end{abstract}

\bigskip

{\bf 2020 MSC:} 47B02; 57N20; 58B10


{\bf Keywords:}  Common complement; complementary subspaces; Grassmann manifold; Banach-Lie group; fiber bundle; contractible space.

\section{Introduction}\label{section 1}
Let $\h$ be a separable complex Hilbert space, and denote by $\mathbb{G}(\h)$ the Grassmann manifold of $\h$ consisting of all the closed subspaces of $\h$.
Consider the set  $\Delta$ of all pairs of subspaces that admit a common complement, which can be described as follows 
$$
\Delta:=\{ (\s,\t) \in \GG(\h) \times \GG(\h) :  \exists \z \in \GG(\h) \hbox{ such that } \z \dot{+} \s =\z \dot{+} \t =\h \}.
$$ 
Here the symbol $\dot{+}$ stands for direct sum ($\s\dot{+}\z=\h$ means that $\s+\z=\h$ and $\s\cap\z=\{0\}$). Several characterizations of this set were given in the seminal work by Lauzon and Treil  \cite{lauzon treil}.
Later on Giol \cite{giol} linked subspaces with common complement and the problem of estimating the number of segments between homotopic idempotents. 
More recently, we proved in \cite{complemento comun} that $\Delta$ is an open  subset of $\GG(\h) \times \GG(\h)$, determined its connected components, and we showed that  its complement  $(\GG(\h) \times \GG(\h))\setminus \Delta$ is a Banach manifold with components described by dimensions and semi-Fredholm indices. In this note we obtain finer results on the topological and differentiable structure of $\Delta$.

We  will use several  geometric constructions related to the Grassmann manifold $\GG(\h)$. This manifold has been thoroughly studied (see, e.g., \cite{cpr, K79, pr}). It has been particularly useful to identify a closed subspace $\s$ with the orthogonal projection $P_\s$ onto $\s$, or with the symmetry $\epsilon_\s=2P_\s-1$ which is the identity on $\s$ and minus the identity in $\s^\perp$. This allows to endow $\GG(\h)$ with the relative topology as a real analytic submanifold of $\b(\h)$, the algebra of bounded linear operators acting on $\h$. Another way to construct a manifold structure on $\GG(\h)$ is to define for each  $\z\in\GG(\h)$, the  following set
$$
\GG^\z:=\{\s\in\GG(\h): \s\dot{+}\z=\h\}.
$$
Subspaces in $\GG^\z$  turn out to be be in one-to-one correspondence with graphs of operators in $\b(\z^\perp, \z)$, and then an atlas can be constructed by defining charts on the family of sets $\{  \GG^\z\}_{\z \in \GG(\h)}$. 
 The presentations of the Grassmann manifold as subspaces or orthogonal projections are indeed the same in the sense of differential geometry, meaning that the  bijective map $\s \to P_\s$ is  a real analytic diffeomorphism (\cite{AM07}).

We now describe our two main results on the structure of $\Delta$. Denote by $Gl(\h)$ the linear group of $\h$ (i.e., bounded linear invertible operators in $\h$) and by $\u(\h)$ the subgroup of unitary operators.
A useful tool in our study will be, for a given closed subspace $\z\subset\h$, the subgroup
$$
Gl^\z:=\{G\in Gl(\h): G(\z)=\z\}.
$$
Clearly, if $\s\dot{+}\z=\h$ and $G\in Gl^\z$, then $G(\s)\dot{+}\z=\h$. This means that  $Gl^\z$ acts on the open sets $\GG^\z \subseteq \GG(\h)$ as follows $G \cdot \s=G(\s)$, for all $\s \in \GG^\z$ and $G \in Gl^\z$. Furthermore, we prove that $Gl^\z$ is a Banach-Lie group and $\GG^\z$ is a real analytic homogeneous space of $Gl^\z$. 
We then construct a real analytic fiber bundle in which the base space is the open set  $\Delta \subseteq \GG(\h) \times \GG(\h)$ and the  total space is given by 
$$
\e:=\bigsqcup_{\z\in\GG(\h)} Gl^\z\times Gl^\z. 
$$
The set $\e$ can be endowed with a manifold structure by using the same ideas of the frame bundle construction in classical differential geometry, and the motivation for considering it comes from the fact that  $(G(\z^\perp),K(\z^\perp)) \in \Delta$ for every $\z \in \GG(\h)$ and $G,K \in Gl^\z$ because $\z^\perp \oplus \z=\h$. This, in turn,  leads us to define the following map
$$
\mathfrak{p}: \e \to \Delta, \, \, \, \mathfrak{p}(Z, G,K)=(G(\z^\perp), K(\z^\perp)).
$$
The first main result of this note  is  that  $\mathfrak{p}$ is a real analytic fiber bundle  (Theorem \ref{pepe fbundle}). Here $\Delta$ is considered with its natural manifold structure as an open subset of $\GG(\h) \times \GG(\h)$.
 

The connected components of $\GG(\h)$ are given by $\GG_{ij}$, where $i,j \in \{0, 1, \ldots , \infty  \}$ are the dimensions and codimensions of the subspaces ($i + j=\infty$); and  
it is known that $\GG_\infty:=\GG_{\infty, \infty}$ is contractible, a fact which follows directly form Kuiper's theorem \cite{kuiper}. 
It was proved in \cite{complemento comun}  that the connected components of $\Delta$  are $\Delta_{ij}:=\GG_{ij} \times \GG_{ij}$ for $i < \infty$ or $j < \infty$, and $\Delta_\infty:=\Delta \cap (\GG_\infty \times \GG_\infty)$. Furthermore, it was shown that $\Delta_\infty$ is dense in $\GG_\infty \times \GG_\infty$. Since $\GG_{i,\infty}$ and $\GG_{\infty,i}$ for $i < \infty$,   have the homotopy type of the classifying space of the finite dimensional linear group $Gl(i)$,   it follows that $\Delta_\infty$ becomes the only component of topological interest. 
The second main result of this note  is  that $\Delta_\infty$ is contractible (Theorem \ref{delta contractil}).


\section{A fiber bundle with base space $\Delta$}

\subsection{The Banach-Lie group $Gl^\z$}\label{sec 1}


We will work in the setting of real analytic Banach manifolds. For the notions of Banach-Lie group, analytic homogeneous space, submanifolds and others we refer to \cite{beltita, Up85}. In particular, we observe that the notion of submanifold we consider also appears in the literature as complemented submanifold. We denote $R(T)$ and $N(T)$  the range and nullspace of an operator $T\in \b(\h)$.  For $\l,\k\subseteq\h$  closed subspaces such that $\l\dot{+}\k=\h$, we denote by $P_{\l\parallel\k}$ the idempotent with range $\l$ and nullspace $\k$. In particular, $P_\s=P_{\s \parallel \s^\perp}$ is the orthogonal projection onto $\s$.

\begin{rem}\label{buckz}
Buckholtz proved in \cite{buckholtz} the following results. 
\begin{enumerate}
\item[i)] Let $P,Q$ be orhogonal projections. Then the following are equivalent:
\begin{enumerate}
\item
$P-Q$ is invertible.
\item
$\|P+Q-1\|<1$.
\item
$R(P)\dot{+}R(Q)=\h$.
\end{enumerate}
\item[ii)] If $\l,\k\subset\h$ are closed subspaces such that $\l\dot{+}\k=\h$, then $P_\l-P_\k$ is invertible, and  
\begin{equation}\label{parallel}
 P_{\l\parallel\k}=P_\l(P_\l-P_\k)^{-1}.
\end{equation}   
\end{enumerate}
\end{rem}

\begin{rem}\label{rem on gz}
As an immediate consequence the sets $\GG^\z=\{\s\in\GG(\h): \s\dot{+}\z=\h\}$, $\z \in \GG(\h)$,
are open in $\GG(\h)$. 
Indeed,  $\s\in\GG^\z$ if and only if $\|P_\s+P_\z-1\|<1$, which is an open condition on $P_\s$. 
Consequently, $\GG^\z$ is a real analytic submanifold of $\b(\h)$ by the identification $\s \mapsto P_\s$. Moreover, the map 
 $ \varphi_ \z : \GG^\z \to \b(\z^\perp, \z)$ defined by $\varphi_\z(\s)=P_\z (P_{\z^\perp}|_\s)^{-1}$,  
is a chart for $\GG(\z)$ at $\z^\perp$, and thus a homeomorphism onto the Banach space $\b(\z^\perp, \z)$ of bounded
linear operators from $\z$ to $\z^\perp$ (see, e.g., \cite{AM07}). Hence the sets $\GG^\z$, $\z \in \GG(\h)$, are also contractible. 
\end{rem}

\begin{prop}
The group $Gl^\z=\{G\in Gl(\h): G(\z)=\z\}$, where $\z \in \GG(\h)$, is a Banach-Lie group. Moreover, it is a  real analytic submanifold of $\b(\h)$.
\end{prop}
\begin{proof}
The proof is straightforward writing the elements of $Gl^\z$ as $2\times 2$ matrices in terms of the decomposition $\h=\z\oplus\z^\perp$: they are the upper triangular elements of $Gl(\h)$.
\end{proof}

The linear group $Gl(\h)$ of $\h$ acts analytically on $\GG(\h)$ in a natural fashion:
$$
G\cdot \s=G(\s), \hbox{ for } G\in Gl(\h), \s\in\GG(\h).
$$
This means that  the map 
$
Gl^\z\times\GG^\z\to \GG^\z, \ (G,\s)\mapsto G(\s)
$
is  real analytic, a fact that follows by the diffeomorphism with the presentation of $\GG(\h)$ as orthogonal projections and the formula  
(\cite{ando}):
$$
P_{G(\s)}=GP_\s G^{-1}\left( GP_\s G^{-1} +(GP_\s G^{-1})^*-1\right)^{-1}.
$$
As we have already observed, given $\z \in \GG(\h)$ the group $Gl^\z$ acts on $\GG^\z$ by $G \cdot \s=G(\s)$, $\s \in \GG^\z$ and $G \in Gl^\z$.  That is, this action is obtained  by  restriction of the above action of $Gl(\h)$ on $\GG(\h)$, and  since $\GG^\z$ and $Gl^\z$ are submanifolds,  we get that  the action of $Gl^\z$  on $\GG^\z$ is also analytic. 
 Fixing the element $\s_0\in\GG^\z$, this action induces the following map 
$$ \pi_{\s_0}:Gl^\z\to \GG^\z, \ \pi_{\s_0}(G)=G(\s_0).$$
Next we see that the action is transitive, or equivalently, the map $\pi_{\s_0}$ is onto. To this end and other differential properties of this map, we introduce the following operators.

\begin{defi}
For $\s, \t, \z \in \GG(\h)$ such that $\z \dot{+}\s=\z \dot{+}  \t=\h$, we write
$$\LL^\z_{\s,\t}:=P_{\z\parallel\s}+P_{\t\parallel\z}P_{\s\parallel\z}.$$
\end{defi} 


\begin{lem}\label{propiedades  de mapa}
For $\s, \t, \z \in \GG(\h)$ as above, we have    $\LL^\z_{\s,\t}(\s)=\t$ and $\LL^\z_{\s,\t} \in Gl^\z$. Furthermore, the map $\GG^\z \times \GG^\z \to Gl^\z$, $(\s,\t)\mapsto \LL^\z_{\s,\t}$ 
 is real analytic.
\end{lem}
\begin{proof}
We set $L=\LL^\z_{\s,\t}$. If $\zeta\in\z$, then $L\zeta=P_{\z \parallel\s}\zeta=\zeta$. If $\xi\in\s$, then $L\xi=P_{\t\parallel\z}\xi\in\t$. Thus $L(\z)=\z$ and $L(\s)\subset\t$. This implies that $N(L)=\{0\}$: if $\chi\in N(L)$, write $\chi=\xi+\zeta$ for $\xi\in\s$ and $\zeta\in\z$. Then $0=L(\xi)+L(\zeta)=P_{\t\parallel\s}\xi + \zeta$, these vectors belong to complementary subspaces, thus $\zeta=0$ and $P_{\t\parallel\z}\xi=0$, and thus also $\xi\in\z$, i.e., $\zeta=0=\xi$. Clearly $\z\subset R(L)$. Note that in fact $L(\s)=\t$: $P_{\t\parallel\z}|_\s:\s\to\t$ is an isomorphism. In particular this implies that $L$ is invertible: $\h=\t+\z\subset R(L)$. Moreover, in view of (\ref{parallel}) we have that 
$$
\LL^\z_{\s,\t}=P_\z(P_\z-P_{\s})^{-1}+P_\t(P_\t-P_\z)^{-1}P_{\s}(P_{\s}-P_\z)^{-1},
$$
which is clearly a real analytic map for $(\s, \t) \in \GG^\z \times \GG^\z$ taking  values in $\b(\h)$. The restriction to $Gl^\z$ is also real analytic, since we have shown that this group is a  submanifold of $\b(\h)$. 
\end{proof}

The main properties relating the open sets $\GG^\z$ and the Banach-Lie group $Gl^\z$ can be stated as follows.

\begin{prop}\label{teo 25}
The map  $\pi_{\s_0}$ is a real analytic submersion, with an analytic global cross section. Moreover, since $\GG^\z$ is a real analytic homogeneous space of $Gl^\z$, the map $\pi_{\s_0}$ is a trivial fiber bundle.
\end{prop} 
\begin{proof}
Given  $\s_0$ and $\t$  closed subspaces in $\GG^\z$, we take the operator 
$
\LL^\z_{\s_0,\t}
$, 
 which satisfies that  $\LL^\z_{\s_0,\t}\in Gl^\z$ and $\pi_{\s_0}(\LL^\z_{\s_0,\t})=\t$. Thus, $\pi_{\s_0}$ is onto. Moreover, 
note have that 
$$
\sigma_{\s_0}:\GG^\z\to Gl^\z, \ \sigma_{\s_0}(\t)=\LL^\z_{\s_0,\t},
$$
is a real analytic global cross section for $\pi_{\s_0}$. 
Therefore, $\pi_{\s_0}:Gl^\z\to\GG^\z$ is a homogeneous space with a global cross section, it follows that it is a trivial bundle (see \cite{beltita}).
\end{proof}    


\subsection{The manifold $\e$}

 For notational simplicity, we let  
$$\l_{\z,\s}:=\LL^\z_{\s,\z^\perp}=P_{\z\parallel\s}+P_{\z^\perp}P_{\s\parallel\z},$$
whenever $\z \in \GG(\h)$ and $\s \in \GG^\z$. Note that $\l_{\z,\s}(\s)=\z^\perp$ and $\l_{\z,\s}\in Gl^\z$. We observe that the map $\GG^\z  \to Gl^\z$, $\s \to \l_{\z, \s}$, is real analytic as a consequence of Lemma  \ref{propiedades  de mapa}

The following auxiliary space will be the total space of our fiber bundle, and will play an important role the study of the topology of $\Delta$. 

\begin{defi}
Put
\begin{align*}
\e & :=\bigsqcup_{\z\in\GG(\h)} Gl^\z\times Gl^\z  =\{(\z,G,K)\in\GG(\h)\times Gl(\h)\times Gl(\h): G(\z)=K(\z)=\z \}.
\end{align*}
\end{defi}

\begin{prop}\label{manifold E}
$\e$ is a real analytic manifold.
\end{prop}
\begin{proof}
Recall that $\GG_{ij}$ are the connected components of $\GG(\h)$, where $i$ and $j$ indicates the dimension and codimension of the subspaces, respectively. Our definition of an atlas depends on those components, so we fix  one component $\GG_{ij}$,  and  two  subspaces $\h_+$ and $\h_-$ such that $\dim \h_-=i$, $\dim \h_+=j$ and $\h_+ \oplus \h_-=\h$.  For a given $\z \in\GG_{ij}$, we denote
$$
\u_{\z}(\h):=\{T\in\u(\h): T(\z^\perp)=\h_+\}.
$$
Then for $\z_0\in\GG_{ij}$, $T_0, T'_0\in\u_{\z_0}(\h)$ we define
$$
\Phi_{\z_0,T_0,T'_0}: \GG^{\z_0}\times Gl^{\h_+}\times Gl^{\h_+} \to \e, 
$$
$$
\Phi_{\z_0,T_0,T'_0}(\z,G,K)=(\z,\l_{\z_0,\z}^{-1}T_0^*GT_0\l_{\z_0,\z}, \l_{\z_0,\z}^{-1}(T'_0)^*KT'_0\l_{\z_0,\z}).
$$
This map is well defined: note that $\l_{\z_0,\z}^{-1}T_0^*GT_0\l_{\z_0,\z}$ is an invertible operator which maps $\z$ onto $\z$:
$$
\z\stackrel{\l_{\z_0,\z}}{\longrightarrow}\z_0^\perp\stackrel{T_0}{\longrightarrow}\h_+\stackrel{G}{\longrightarrow}\h_+\stackrel{T_0^*}{\longrightarrow}\z_0^\perp\stackrel{\l_{\z_0,\z}^{-1}}{\longrightarrow} \z,
$$
and similarly for the last coordinate of $\Phi_{\z_0,T_0,T'_0}(\z,G,K)$.
Note also that $\Phi_{\z_0,T_0,T'_0}$ is injective:
$$
\Phi_{\z_0,T_0,T'_0}(\z,G,K)=\Phi_{\z_0,T_0,T'_0}(\z',G',K'),
$$
clearly implies that  $\z=\z'$, from which $G=G'$ and $K=K'$ readily follow. This map has a local inverse, defined  on the set
$$
\b_{\z_0} := \{(\s,A,B)\in\e: \|P_\s-P_{\z_0}\|<1\}.
$$
If $(\s,A,B)\in\b_{\z_0}$ then $\s\dot{+}\z_0^\perp=\h$, so that $\s,\z_0\in\GG^{\z_0^\perp}$, and $\l_{\z_0,\s}$ is defined, and the map $\Phi_{\z_0,T_0,T'_0}$ can be clearly reversed. Thus we have local charts for $\e$, which model it in the Banach space $\b(\h)^3$. 

Let us check that  the transition maps are analytic. Pick $\z_0, \z_1\in\GG_{ij}$.  Denote $\Phi_1=\Phi_{\z_1,T_1,T'_1}$ and $\Phi_0=\Phi_{\z_0,T_0,T'_0}$, then
$$
\Phi_1^{-1}\circ\Phi_0:(\GG^{\z_0}\cap\GG^{\z_1})\times Gl^{\h_+}\times Gl^{\h_+}\to (\GG^{\z_0}\cap\GG^{\z_1})\times Gl^{\h_+}\times Gl^{\h_+}
$$
$$
\Phi_1^{-1}\circ\Phi_0(\z,G,K)=\Phi^{-1}(\z,\l_{\z_0,\z}^{-1}T_0^*GT_0\l_{\z_0,z},\l_{\z_0,\z}^{-1}(T'_0)^*KT'_0\l_{\z_0,z})=
$$
$$
\left(\z,T_1\l_{\z_1,\z}(\l_{\z_0,\z}^{-1}T_0^*GT_0\l_{\z_0,z})\l_{\z_1,\z}^{-1}T_1^*,
T'_1\l_{\z_1,\z}(\l_{\z_0,\z}^{-1}(T'_0)^*KT'_0\l_{\z_0,z})\l_{\z_1,\z}^{-1}(T'_1)^*\right),
$$
which is analytic. This completes the construction of an atlas for $\e$. 
\end{proof}

\begin{rem}\label{open and others}
The idea for constructing an atlas above is adapted from the frame bundle construction 
in classical differential geometry (see for instance \cite{HN12}). We further remark the following facts regarding the previous proof.

\medskip

\noi $i)$ The charts define a topology on $\e$ in the usual fashion. In particular, the sets 
of the form 
\begin{equation*}
\b_{\z_0} = \{(\s,A,B)\in\e: \|P_\s-P_{\z_0}\|<1\}=\bigsqcup_{\z \in \GG^{\z_0^\perp}} Gl^\z \times Gl^\z
\end{equation*}
 that appear in the proof are open in $\e$, and  the family  $\{ \b_\z \}_{\z \in \GG(\h)}$  is an open covering of $\e$. Also the connected components of $\e$ are given by
$$
\e_{ij}:=\bigsqcup_{\z\in\GG_{ij}} Gl^\z\times Gl^\z,
$$
where $i,j \in \{ 0, 1, \ldots, \infty\}$ and $i + j=\infty$. We sometimes abbreviate $\e_\infty:=\e_{\infty,\infty}$.

\medskip

\noi $ii)$ We have shown in fact that $\e$ is a  submanifold of $\GG(\h)\times\b(\h)^2$ (and since $\GG(\h)$ is a submanifold of $\b(\h)$ \cite{cpr, pr}, we find that $\e$ is also a  submanifold of $\b(\h)^3$). Indeed, we have modeled $\e$ using local charts with values in $\GG^{\z_0}\times (Gl^{\h_+})^2$. As we saw earlier, $\GG^{\z_0}$ is open in $\GG(\h)$, and $Gl^{\h_+}$ is a  submanifold of $\b(\h)$.
\end{rem}

\subsection{A real analytic fiber bundle}

Related to $\e$, we have the following map
\begin{equation*}
\pepe:\e\to\Delta, \,\, \,  \, \pepe(\z,G,K)=(G(\z^\perp),K(\z^\perp)).
\end{equation*}
Note that $\pepe$ is well defined: since $G,K\in Gl^\z$, $\z\oplus\z^\perp=\h$ implies that 
$\z\dot{+}G(\z^\perp)=\h=\z\dot{+}K(\z^\perp)$. Also observe that $\pepe:\e\to\Delta$ is onto. Given $(\s,\t)\in\Delta$, let $\z$ be a common complement for $\s$ and $\t$. Then $\s,\t,\z^\perp$ have common complement $\z$, which enables the existence of $\LL_{ \z^\perp, \s}^\z$ and $\LL_{ \z^\perp ,\t}^\z$ by Lemma \ref{propiedades  de mapa}, and clearly $\pepe(\z,\LL_{ \z^\perp, \s}^\z, \LL_{ \z^\perp ,\t}^\z)=(\s,\t)$. Recall from \cite{complemento comun} that the connected components of $\Delta$ are given by 
$\Delta_{ij}=\GG_{ij} \times \GG_{ij}$ for $i < \infty$ or $j < \infty$, and $\Delta_\infty:=\Delta_{\infty,\infty}:=\Delta \cap (\GG_\infty \times \GG_\infty)$. We thus get that $\pepe(\e_{ij})=\Delta_{ji}$  for all $i$, $j$ such that $i + j=\infty$.

Let us check now that  $\pepe$ is a real analytic fiber bundle. We examine first the fibers: for $(\s_0,\t_0)\in\Delta$, set
$$
\FF_{(\s_0,\t_0)}:=\pepe^{-1}(\s_0,\t_0)=\{(\z,G,K)\in\e: G(\z^\perp)=\s_0, K(\z^\perp)=\t_0\}.
$$
 We have the following characterization.
\begin{prop}\label{prop 48}
Take $(\s_0,\t_0) \in \Delta_{ij}$ and  two  subspaces $\h_+ $, $\h_-$ such that $\dim \h_-=i$, $\dim \h_+=j$ and $\h_+ \oplus \h_-=\h$. Then $\FF_{(\s_0,\t_0)}$ is a closed submanifold of $\e$, and there is a diffeomorphism
\begin{equation*}
\FF_{(\s_0,\t_0)}\simeq \GG^{\s_0} \times (Gl^{\h_+}\cap Gl^{\h_-})^2.
\end{equation*}
\end{prop}
\begin{proof}
An alternative expression for the fibers is given by
\begin{equation*}\label{another fiber exp}
\FF_{(\s_0 , \t_0)}=\{  (\z, G,K) \in  \GG^{\s_0} \times Gl^\z \times Gl^\z :  G(\z^\perp)=\s_0, \, K(\z^\perp)=\t_0  \}.
\end{equation*} 
Indeed, for $(\z,G,K) \in \e$ such that $\pepe(\z,G,K)=(\s_0 , \t_0)$, one has $G(\z^\perp)=\s_0$ and $K(\z^\perp)=\t_0$. Since $\h=\z\oplus \z^\perp$, it follows that $\h=\z\dot{+}G(\z^\perp)=\z \dot{+} \s_0$, which means $\z \in \GG^{\s_0}$. The reversed inclusion is trivial. In particular, note 
$$
\FF_{(\s_0, \t_0)} \subseteq \bigsqcup_{\z \in \GG^{\s_0}} Gl^\z \times Gl^\z=\b_{\s_0^{\perp}}, 
$$
where the set on the  right-hand side is open in $\e$  (see Remark \ref{open and others}).

 Next recall that for $\z$ such that $\s_0 \dot{+} \z=\t_0 \dot{+} \z=\h$, there exists $\LL_\z:=\LL^\z_{\t_0,\s_0} \in Gl^\z$ such that $\LL_\z(\t_0)=\s_0$. Fix $T \in \u(\h)$ such that $T\s_0=\h_-$.  We consider the map $\Psi:\GG^{\s_0} \times (Gl^{\h+})^2 \to \b_{\s_0^\perp}$, $\Psi(\z,G,K)=(\Psi_1(\z,G,K), \Psi_2(\z,G,K) , \Psi_3(\z,G,K))$ defined by
 \begin{align*}
\Psi_1(\z,G,K) & =  \z ;\\
\Psi_2(\z,G,K) & =  \l_{\s_0, \z}^{-1} T^* G^{-1}T \l_{\s_0, \z} \l_{\z , \s_0}^{-1}  ;\\
\Psi_3(\z,G,K) & =  \LL_{\z}^{-1} \l_{\s_0, \z}^{-1} T^* K^{-1}T  \l_{\s_0, \z} \l_{\z, \s_0}^{-1} .  
\end{align*}
First we show that $\Psi$ is well defined.  Its second coordinate   takes values on $Gl^{\z}$, since
 $$
\z \stackrel{\l_{\z,\s_0}^{-1}}{\longrightarrow}\z
\stackrel{\l_{\s_0,\z}}{\longrightarrow}\s_0^\perp
\stackrel{T}{\longrightarrow}\h_+
\stackrel{G^{-1}}{\longrightarrow} \h_+
\stackrel{T^*}{\longrightarrow}\s_0^\perp
\stackrel{\l_{\s_0, \z}^{-1}}{\longrightarrow}\z.
$$
From this expression we also deduce that its third coordinate takes values on $Gl^\z$. Clearly,   $\Psi$ is bijective because one can easily write its inverse.  Take $\Phi_{\z_0, T_0,T_0'}$ a chart on $\e$ defined on $\GG^{\z_0} \times (Gl^{\h+})^2$ satisfying $\GG^{\s_0} \cap \GG^{\z_0} \neq \emptyset$. A similar computation to that at the end of the proof of Theorem \ref{manifold E} shows that $\Psi^{-1} \circ \Phi_{\z_0, T_0,T_0'}$ and  $\Phi_{\z_0, T_0,T_0'}^{-1} \circ \Psi$ are analytic functions. Thus, $\Psi$ is  also a chart compatible with the manifold structure of $\e$. 

Now note that for $G, K \in  Gl^{\h_-}$, we have
$$
\z^\perp \stackrel{\l_{\z,\s_0}^{-1}}{\longrightarrow}\s_0
\stackrel{\l_{\s_0,\z}}{\longrightarrow}\s_0
\stackrel{T}{\longrightarrow}\h_-
\stackrel{G^{-1}}{\longrightarrow} \h_-
\stackrel{T^*}{\longrightarrow}\s_0
\stackrel{\l_{\s_0, \z}^{-1}}{\longrightarrow}\s_0
$$
and
$$
\z^\perp \stackrel{\l_{\z,\s_0}^{-1}}{\longrightarrow}\s_0
\stackrel{\l_{\s_0,\z}}{\longrightarrow}\s_0
\stackrel{T}{\longrightarrow}\h_-
\stackrel{K^{-1}}{\longrightarrow} \h_-
\stackrel{T^*}{\longrightarrow}\s_0
\stackrel{\l_{\s_0, \z}^{-1}}{\longrightarrow}\s_0
\stackrel{\LL_{\z}^{-1}}{\longrightarrow}\t_0.
$$
Thus, we find that $\Psi(\GG^{\s_0} \times (Gl^{\h+} \cap Gl^{\h_-})^2)=\FF_{(\s_0 , \t_0)}$. Since $Gl^{\h+} \cap Gl^{\h_-}$ and $Gl^{\h_+}$ are block diagonal and block upper triangular operators with respect to the decomposition $\h_+ \oplus \h_-=\h$, it follows that $\GG^{\s_0} \times (Gl^{\h+} \cap Gl^{\h_-})^2$ is a submanifold of $\GG^{\s_0} \times (Gl^{\h+} )^2$.  Hence we have proved that $\FF_{(\s_0, \t_0)}$ is a submanifold of $\e$ and the stated diffeomorphism. It is straightforward to check that  $\FF_{(\s_0, \t_0)}$ is closed in $\e$.
\end{proof}

Now we need to recall a well-known construction regarding two projections (for details see for instance \cite{pr}). Take $P$, $Q$ two orthogonal projections such that $\|  P - Q \| <1$. From this last estimate, it can be shown that $S=QP + (1-Q)(1-P)$ is invertible, and the unitary $W=S|S|^{-1}$ satisfies $WPW^*=Q$.

\begin{defi}
Let $\s, \z \in \GG(\h)$ such that $\s\dot{+}\z=\h$, which  in terms of projections is equivalent to $\|P_\s-P_{\z^\perp}\|<1$. We denote by $W^{\z^\perp}_\s$ the unitary satisfying  $W_\s^{\z^\perp}P_\s (W_\s^{\z^\perp})^*=P_{\z^\perp}$ given by the above construction.
\end{defi}

Clearly,  
the map $\GG^\z \to \u(\h)$, $\s \mapsto W_\s^{\z^\perp}$ is real analytic by properties of the analytic functional calculus.

Our main result of this section now follows.

\begin{teo}\label{pepe fbundle}
The map $\pepe$ is a real analytic fiber bundle.
\end{teo}
\begin{proof}
We have already observed above that $\pepe$ is well defined and onto. Clearly, the map $\pepe$ is analytic by using a local expression in  the charts of $\e$. 
Let us exhibit local trivializations for $\pepe$. Fix $(\s_0,\t_0)\in\Delta$, and a common complement $\z_0$ for $\s_0,\t_0$. Consider the open neighborhood 
$$
(\s_0,\t_0)\in\Delta^{\z_0}=\{(\s,\t)\in\Delta: \s\dot{+}\z_0=\h=\t\dot{+}\z_0\}.
$$
 Pick $(\z,G,K)\in\pepe^{-1}(\Delta^{\z_0})$. Denote $\s=G(\z^\perp)$ and $\t=K(\z^\perp)$. Since $(\s,\t)\in\Delta^{\z_0}$, we have that $\|P_\s-P_{\z_0^\perp}\|<1$ and $\|P_\t-P_{\z_0^\perp}\|<1$. Therefore, in particular,  there exists $W_\s^{\z_0^\perp} \in \u(\h)$ such that $W_\s^{\z_0^\perp}(\s)=\z_0^\perp$. Let us fix $T\in\u(\h)$ such that $T(\z_0^\perp)=\h_-$.  As usual $\h_+$, $\h_-$ are fixed subspaces such that $\h_- \oplus \h_+=\h$ of the appropriate dimensions determined by $(\s_0 , \t_0)$. Note also that since $G,K\in Gl^\z$, $\z\oplus\z^\perp=\h$ implies that $G(\z)\dot{+}G(\z^\perp)=\z\dot{+}\s=\h$ and $K(\z)\dot{+}K(\z^\perp)=\z\dot{+}\t=\h$. Then 
$$
\h=TW_\s^{\z_0^\perp}(\z)\dot{+}TW_\s^{\z_0^\perp}(\s)=TW_\s^{\z_0^\perp}(\z)\dot{+}\h_-,
$$
which means
\begin{equation}\label{twz}
TW_\s^{\z_0^\perp}(\z)\in \GG^{\h_-}.
\end{equation}

Let us define the local trivialization at the neighbourhood $\Delta^{\z_0}$. The idea is that the definition of $\Delta^{\z_0}$ allows us to define  $ W_\s^{\z_0^\perp}$ to move  a fiber $\FF_{(\s, \t)}$ to the generic fiber
$$
\FF:=\GG^{\h_-}\times \left(Gl^{\h_+}\cap Gl^{\h_-}\right)^2 \simeq \FF_{(\s_0, \t_0)},
$$
which is a diffeomorphism by Proposition \ref{prop 48}. Consider
$$
\Phi:\pepe^{-1}(\Delta^{\z_0})\to \Delta^{\z_0}\times \FF,
$$
$$
 \Phi(\z,G,K)=(G(\z^\perp),K(\z^\perp), u(\z,G,K), a(\z,G,K), b(\z,G,K)).
$$
Again we abbreviate $\s=G(\z^\perp)$ and $\t=K(\z^\perp)$ to define the last coordinates:
\begin{align*}
  u(\z,G,K) & =TW_\s^{\z_0^\perp}(\z); \\
a(\z,G,K) & =T W_\s^{\z_0^\perp} \l_{\s,\z} \l_{\z,\s}^{-1}G^{-1}\l_{\s,\z}^{-1} (W_\s^{\z_0^\perp})^*T^*; \\
b(\z,G,K) & =TW_\s^{\z_0^\perp}\l_{\s,\z} \l_{\z,\s}^{-1}K^{-1}(\LL^\z_{\t,\s})^{-1} \l_{\s,\z}^{-1} (W_\s^{\z_0^\perp})^*T^*.
\end{align*}

The first coordinate map clearly belongs to $\GG^{\h_-}$ as remarked in (\ref{twz}). Let us see  that $a$ and $b$ are also well defined. Note that under the current conditions, we have 
$$
\z\dot{+}\s=\z_0\dot{+}\s=\z\dot{+}\t=\z_0\dot{+}\t=\h.
$$
Also 
$
\s\dot{+}\z=\h$ if and only if $\|P_\s+P_\z-1\|<1$. This is equivalent to $\|(1-P_\s)+(1-P_\z)-1\|<1$, or $\s^\perp\dot{+}\z^\perp=\h$. Similarly for the other direct sum decompositions. This implies that  the $\l$ and $\LL$ operators used above are well defined.  In particular $\l_{\z,\s}\in Gl^\z$ which maps $\s$ onto $\z^\perp$, and $\l_{\s,\z}\in Gl^\s$ which maps $\z$ onto $\s^\perp$. 

Next let us check that $a$ and $b$ belong to $Gl^{\h_+}\cap Gl^{\h_-}$, i.e., that they map $\h_+$ onto $\h_+$ and $\h_-$ onto $\h_-$.
For the case of $a$:
$$
\h_-\stackrel{T^*}{\longrightarrow}\z_0^\perp\stackrel{(W_\s^{\z_0^\perp})^*}{\longrightarrow}\s\stackrel{\l_{\s,\z}^{-1}}{\longrightarrow}\s\stackrel{G^{-1}}{\longrightarrow}\z^\perp\stackrel{\l_{\z,\s}^{-1}}{\longrightarrow}\s\stackrel{\l_{\s,\z}}{\longrightarrow}\s\stackrel{W_\s^{\z_0^\perp}}{\longrightarrow}\z_0^\perp\stackrel{T}{\longrightarrow}\h_-.
$$
and 
$$
\h_+\stackrel{T^*}{\longrightarrow}\z_0\stackrel{(W_\s^{\z_0^\perp})^*}{\longrightarrow}\s^\perp \stackrel{\l_{\s,\z}^{-1}}{\longrightarrow}\z\stackrel{G^{-1}}{\longrightarrow}\z\stackrel{\l_{\z,\s}^{-1}}{\longrightarrow}\z\stackrel{\l_{\s,\z}}{\longrightarrow}\s^\perp\stackrel{W_\s^{\z_0^\perp}}{\longrightarrow}\z_0\stackrel{T}{\longrightarrow}\h_+.
$$
The verification that the map $b$ is diagonal is similar. The map $\Phi$ can be clearly reversed, once it is well defined in the appropriate neighbourhoods. It is a real analytic diffeomorphism, which clearly trivializes the map $\pepe$.
\end{proof}

\section{Topology of $\Delta$}\label{topology}

Another map related to the manifold $\e$ is the following
$$
\pi: \e \to \GG(\h), \, \, \, \pi(\z, G,K)=\z. 
$$
Let us prove that the map $\pi$ is a also a real analytic fiber bundle. For a subspace $\z_0\in\GG_{ij} \subseteq \GG(\h)$, fix two subspaces $\h_-$ and $\h_+$ such that $\dim \h_-=j$, $\dim \h_+=i$  and $\h_-\oplus \h_+=\h$. Then, 
$$
\pi^{-1}(\z_0)=\{(\z_0,G,K): G,K\in Gl^{\z_0}\}\simeq Gl^{\z_0}\times Gl^{\z_0}\simeq Gl^{\h_+}\times Gl^{\h_+}.
$$

\begin{prop}
The map $\pi$ is a real analytic fiber bundle.
\end{prop}
\begin{proof}
Take arbitrary $i$, $j$ such that $i + j=\infty$ and fixed subspaces $\h_+$ and $\h_-$ as above. We saw that the fibers are Banach-Lie groups isomorphic to $Gl^{\h_+}\times Gl^{\h_+}$. Let us construct local trivializations for $\pi$. Fix $\z_0\in\GG_{ij}$,  and consider the neighborhood 
$$
\z_0\in\GG^{\z_0^\perp}=\{\s\in\GG_{ij}: \|P_\s-P_{\z_0}\|<1\}.
$$
This allows the existence of a real analytic map $\GG^{\z_0^\perp} \ni\s\mapsto W_\s^{\z_0}\in\u(\h)$ which satisfies  $W_\s^{\z_0}(\s)=\z_0$. Then we have the diffeomorphism
$$
\pi^{-1}(\GG^{\z_0^\perp})\to \GG^{\z_0^\perp} \times \pi^{-1}(\z_0), \ (\s,G,K)\mapsto \s\times(\z_0, W_\s^{\z_0}G (W_\s^{\z_0})^*, W_\s^{\z_0}K (W_\s^{\z_0})^*),
$$
which trivializes $\pi$. Note that the above map is well defined, since clearly 
$$
W_\s^{\z_0}G (W_\s^{\z_0})^*, W_\s^{\z_0}K (W_\s^{\z_0})^*\in Gl^{\z_0},
$$
and it can be easily reversed.
\end{proof}

As we have observed at the end of Section \ref{section 1}, we are interested in topology of the connected component $\Delta_\infty=\Delta \cap(\GG_\infty \times \GG_\infty)$. It can be studied by considering  infinite and co-infinite dimensional subspaces in the previous constructions. We begin with the connected component 
$\e_\infty$ of $\e$ (see Remark \ref{open and others}). We have the following consequence.

\begin{coro}\label{coro 47}
 $\e_\infty$ is contractible.
\end{coro}
\begin{proof}
The manifold $\e_\infty$ is the total space of the fiber bundle $\pi|_{\e_\infty} : \e_\infty \to \GG_\infty$ with contractible fibers isomorphic to $Gl^{\h_+} \times Gl^{\h_+}$ ($\dim \h_{\pm}=\infty$) and contractible base space $\GG_\infty$ by Kuiper's theorem. Therefore it has trivial homotopy groups,  and being a manifold modeled in a Banach space, is then contractible (see \cite{palais}).
\end{proof}

Our main result on the topology of $\Delta_\infty$ follows.
\begin{teo}\label{delta contractil}
The open dense subset $\Delta_\infty\subset \GG_\infty \times \GG_\infty$ is contractible.
\end{teo}
\begin{proof}
The fiber bundle $\pepe|_{\e_\infty}:\e_\infty \to\Delta_\infty$ has contractible total space (Corollary \ref{coro 47}). From Proposition \ref{prop 48} we know that for any $(\s_0, \t_0) \in \Delta_\infty$ the fiber $\FF_{(\s_0, \t_0)}$ is diffeomorphic to $\GG^{\s_0} \times (Gl^{\h_+}\cap Gl^{\h_-})^2$. Recall that $\GG^{\s_0}$ is contractible by Remark \ref{rem on gz}, and $Gl^{\h_+}\cap Gl^{\h_-}\simeq Gl(\h_+) \times Gl(\h_-)$ is also contractible by Kuiper's theorem  ($\dim \h_{\pm}=\infty$). Thus, $\pepe|_{\e_\infty}$ has also contractible  fibers.  We thus get that  $\Delta_\infty$ is a Banach manifold with  trivial homotopy type, and hence it is contractible \cite{palais}.
\end{proof}

Also we have the following consequence. 
\begin{coro}
The maps $\pepe|_{\e_\infty}$ and $\pi |_{\e_\infty}$ are in fact trivial bundles.
\end{coro}
\begin{proof}
Both bundles have contractible base space, total space and fibers.
\end{proof}
\noi 
\textbf{Acknowledgment.}
This work was supported by the grant PICT 2019 04060 (FONCyT - ANPCyT, Argentina), (PIP 2021/2023 11220200103209CO),  ANPCyT (2015 1505/ 2017 0883) and FCE-UNLP (11X974).

{\small

\noi E. Andruchow, {\sc  {Instituto Argentino de Matem\'atica, `Alberto P. Calder\'on', CONICET, Saavedra 15 3er. piso,
(1083) Buenos Aires, Argentina }} and {\sc Universidad Nacional de General Sarmiento, J.M. Gutierrez 1150, (1613) Los Polvorines, Argentina}

\noi  e-mail: eandruch@campus.ungs.edu.ar

\medskip

\noi E. Chiumiento, {\sc  {Instituto Argentino de Matem\'atica, `Alberto P. Calder\'on', CONICET, Saavedra 15 3er. piso,
}(1083) Buenos Aires, Argentina}  and {\sc Departamento de Matem\'aica y Centro de Matem\'atica La Plata, Universidad Nacional de La Plata, Calle 50 y 115 (1900) La Plata, Argentina}

\noi  e-mail: eduardo@mate.unlp.edu.ar

}


\begin{thebibliography}{XX}

\bibitem{AM07}   Abbati, M.C.; Mani\`a, A.,  A geometrical setting for geometric phases on complex Grassmann manifolds, J.
Geom. Phys. 57 (2007), 777--797.


\bibitem{ando} Ando, T., Unbounded or bounded idempotent operators in Hilbert space. Linear Algebra Appl. 438 (2013), no. 10, 3769--3775. 





\bibitem{complemento comun} Andruchow, E; Chiumiento, E., Subspaces with or without a common complement (arXiv
 2412.18113).





\bibitem{buckholtz}  Buckholtz, D, Hilbert space idempotents and involutions. Proc. Amer. Math. Soc. 128 (2000), no. 5, 1415--1418.





 

\bibitem{beltita} Belti$\c{t}$$\breve{\text{a}}$, D., Smooth homogeneous structures in operator theory. Chapman \& Hall/CRC Monographs and Surveys in Pure and Applied Mathematics 137, Boca Raton, 2006.



\bibitem{cpr} Corach, G.; Porta, H.; Recht, L., The geometry of spaces of projections in $C^*$-algebras, Adv. Math. 101 (1993), no. 1, 59--77.














\bibitem{giol} Giol, J., Segments of bounded linear idempotents on a Hilbert space, J. Funct. Anal. 229 (2005), no. 2, 405--423. 





\bibitem{HN12} Hilgert J.; Neeb, K.-H., Structure and Geometry of Lie Groups, Springer, New York, 2012.


\bibitem{K79} Kovarik, Z.V., Manifolds of linear involutions, Linear Algebra Appl. 24 (1979), 271--287.




\bibitem{kuiper} Kuiper, N. H., The homotopy type of the unitary group of Hilbert space, Topology 3 (1965), 19--30.

\bibitem{lauzon treil} Lauzon, M.; Treil, S., Common complements of two subspaces of a Hilbert space, J. Funct. Anal. 212 (2004), no. 2, 500--512.




\bibitem{palais} Palais, R. S., Homotopy theory of infinite dimensional manifolds,  Topology 5 (1966), 1--16.



\bibitem{pr} Porta, H.; Recht, L., Minimality of geodesics in Grassmann manifolds, Proc. Amer. Math. Soc. 100 (1987), 464--466.






\bibitem{Up85} Upmeier, H.,  Symmetric Banach manifolds and Jordan C*-algebras.  
North-Holland Math. Stud. 104, Notas de Matem\'atica 96, North-Holland, Amsterdam,
1985.


\end{thebibliography}
\end{document}